\documentclass[10pt]{amsart}

\usepackage{color}
\usepackage{pictexwd}
\usepackage{graphicx}
\usepackage{hyperref}
\usepackage[all,dvips,arc,curve,color,frame]{xy}

\newtheorem{theorem}{Theorem}[section]
\newtheorem{lemma}[theorem]{Lemma}
\theoremstyle{definition}
\newtheorem{corollary}[theorem]{Corollary}

\newtheorem{example}[theorem]{Example}

\newtheorem{proposition}[theorem]{Proposition}

\newtheorem*{theorem*}{Theorem}
\newtheorem*{proposition*}{Proposition}

\DeclareMathOperator{\im}{{\bf{}im}}
\DeclareMathOperator{\rwr}{{\bf{}wr}}
\DeclareMathOperator{\Cay}{{\bf{}Cay}}
\DeclareMathOperator{\Sch}{{\bf{}Sch}}
\DeclareMathOperator{\SSP}{{\bf{}SSP}}
\DeclareMathOperator{\AGP}{{\bf{}AGP}}
\DeclareMathOperator{\supp}{{\bf{}supp}}
\DeclareMathOperator{\code}{{code}}
\DeclareMathOperator{\ncl}{{\bf ncl}}

\def\Ga{{\Gamma}}

\def\ovc{{\overline{c}}}
\def\ovv{{\overline{v}}}
\def\ovu{{\overline{u}}}
\def\ovw{{\overline{w}}}
\def\ovx{{\overline{x}}}

\def\tO{{\tilde O}}
\def\P{{\mathbf{P}}}
\def\NP{{\mathbf{NP}}}

\def\WP{{\mathbf{WP}}}
\def\PP{{\mathbf{PP}}}
\def\CPR{{\mathbf{CP}}}
\def\CC{{\mathcal C}}
\def\CF{{\mathcal F}}
\def\CN{{\mathcal N}}

\newcommand{\gp}[1]{{\left\langle #1 \right\rangle}}
\newcommand{\gpr}[2]{{\left\langle #1 \mid #2 \right\rangle}}
\newcommand{\rb}[1]{{\left( #1 \right)}}

\def\MN{{\mathbb{N}}}
\def\MZ{{\mathbb{Z}}}

\newcommand{\Set}[2]{\left\{\, #1 \;\middle|\; #2 \,\right\}}

\newtheorem{innercustomthm}{Theorem}
\newenvironment{customthm}[1]
  {\renewcommand\theinnercustomthm{#1}\innercustomthm}
  {\endinnercustomthm}
\newtheorem{innercustomco}{Corollary}
\newenvironment{customco}[1]
  {\renewcommand\theinnercustomco{#1}\innercustomco}
  {\endinnercustomco}

\let\oldmarginpar\marginpar
\renewcommand\marginpar[1]{\-\oldmarginpar[\raggedleft\footnotesize #1]%
{\raggedright\footnotesize #1}}

\title{Magnus embedding and algorithmic properties of groups $F/N^{(d)}$}
\date{\today}

\author[F. Gul]{Funda Gul}
\author[M. Sohrabi]{Mahmood Sohrabi}

\author[A. Ushakov]{Alexander Ushakov}
\address{Department of Mathematics, Stevens Institute of Technology, Hoboken, NJ, USA}
\email{msohrabi,fgul,aushakov@stevens.edu}
\thanks{The third author has been partially supported by NSA Mathematical Sciences Program grant number H98230-14-1-0128}

\begin{document}

\begin{abstract}
In this paper we further study properties of Magnus embedding, give a precise
reducibility diagram for Dehn problems in groups of the form $F/N^{(d)}$,
and provide a detailed answer to Problem 12.98 in Kourovka notebook.
We also show that most of the reductions are polynomial time reductions
and can be used in practical computation.


\end{abstract}

\maketitle

\section{Introduction}

Let $F=F(X)$ be the free group on generators $X$, $N$ a normal subgroup of $F$,
$N'$ the derived subgroup of $N$,
and $N^{(d)}$ the $d$th derived subgroup of $N$.
The Magnus embedding is the main tool to study groups of the form $F/N'$.
It was introduced in \cite{Magnus:1939} by W.~Magnus
who showed that the elements of $F/N'$ can be encoded by $2\times 2$ matrices:
$$
M(X;N) =
\Set{\left(
\begin{array}{cc}
g & \pi\\
0 & 1
\end{array}
\right)
}{g\in F/N,~\pi\in \CF_\Gamma},
$$
where $\CF_\Gamma$ is a free module over the group ring $\MZ F/N$.
In this paper we study algorithmic properties of Magnus embedding and
decidability of the following problems for groups of the form $F/N^{(d)}$.

\medskip\noindent
{\bf $\WP(G)$, word problem in $G$.} Given a word $w$ in the generators of $G$ decide if $w=1$ in $G$, or not.

\medskip\noindent
{\bf $\CPR(G)$, conjugacy problem in $G$.} Given words $u,v$ in the generators of $G$ decide if $u\sim v$ in $G$, or not.

\medskip\noindent
{\bf $\PP(G)$, power problem in $G$.} Given words $u,v$ in the generators of $G$ decide if $v=u^k$ in $G$ for some $k\in\MZ$, or not.


\medskip
The Magnus embedding proved to be especially robust in the study of free solvable groups.
Indeed, free solvable group naturally appear in the context because
$$F/F^{(d)} = F/(F^{(d-1)})'$$
is the free solvable group of rank $n$ and degree $d$.
It immediately follows from the work of Magnus that decidability of
the word problem in $F/N$ implies decidability of the word problem in $F/N'$.
The conjugacy problem in groups of the type $F/N'$ was first approached by J.~Matthews in \cite{Matthews:1966}
who proved that:
\begin{itemize}
\item[(a)]
$u,v\in F/N'$ are conjugate (for free abelian $F/N$) if and only if their images under Magnus embedding are conjugate in $M(X;N)$;
\item[(b)]
conjugacy problem in $M(X;N)$ is decidable if and only if conjugacy problem in $F/N$ is decidable and
power problem in $F/N$ is decidable.
\end{itemize}
These two facts imply that free metabelian groups have decidable conjugacy problem.
Later Remeslennikov and Sokolov in \cite{Remeslennikov-Sokolov:1970} extended (a) to any torsion free group $F/N$, showed that the power problem is decidable
in free solvable groups, and deduced that free solvable groups have decidable conjugacy problem.
Finally, C.~Gupta in \cite{Gupta:1982} proved that (a) holds for groups with torsion as well
and hence decidability of the conjugacy and power problems
in $F/N$ implies decidability of the conjugacy problem in $F/N'$.
We use the following notation for reducibility of decision problems
in the sequel:
$$
\left\{
\begin{array}{l}
\CPR(F/N) \\
\PP(F/N)
\end{array}
\right.
\Rightarrow \CPR(F/N').
$$

In the light of these results, V. Shpilrain raised the following questions
in \cite[Problem 12.98]{Kourovka}.
Is it correct that:
\begin{itemize}
\item[(a)]
$\WP(F/N)$ is decidable if and only if $\WP(F/N')$ is decidable.
\item[(b)]
$\CPR(F/N)$ is decidable if and only if $\CPR(F/N')$ is decidable.
\item[(c)]
$\WP(F/N')$ is decidable if and only if $\CPR(F/N')$ is decidable.
\end{itemize}
It was shown by Anokhin in \cite{Anokhin:1997} that only 12.98(a) has an affirmative answer.
He constructed a group $F/N$ such that $\CPR(F/N)$ is decidable and $\CPR(F/N')$ is undecidable.
Such a group is clearly a counterexample to both 12.98(b) and 12.98(c).

We also would like to mention several results related to practical computations
in free solvable groups in which the Magnus embedding plays a crucial role.
S.~Vassileva showed in \cite{Vassilieva:2011} that the power problem in free solvable groups
can be solved in $O(rd (|u|+|v|)^6)$ time and used that result to
show that the Matthews-Remeslennikov-Sokolov
approach can be transformed into a polynomial time $O(rd(|u|+|v|)^8)$ algorithm for the conjugacy problem.
In \cite{Ushakov:2014} those complexity bounds were further improved and randomized algorithms were developed.
Another generalization was done by Lysenok and Ushakov in \cite{Lysenok-Ushakov:2014}.
It was shown that the Diophantine problem for spherical quadratic equations, i.e., equations of the form:
$$
z_1^{-1}c_1z_1 \ldots z_k^{-1}c_kz_k =1,
$$
in free metabelian groups is decidable.
Recall that for every $n\ge 2$ the Diophantine problem in free metabelian groups is undecidable (see \cite{Roman'kov_1979}).
Recently Vassileva proved in \cite{Vassilieva:2014} that the Magnus embedding is a quasi-isometry.

\subsection{Our contribution}

Here we shortly outline the main results of the paper
(somewhat simplifying the statements).
Let $F$ be a free group of rank at least $2$ and $N$ a recursively enumerable
normal subgroup of $F$.


\begin{customthm}{\ref{th:Magnus_PP}}
$\WP(F/N) \Rightarrow \PP(F/N')$.
\end{customthm}


\begin{customthm}{\ref{th:Magnus_CP}}
$\PP(F/N) \Rightarrow \CPR(F/N')$.
\end{customthm}

\begin{customthm}{\ref{th:CP_PP}}
$\PP(F/N) \Leftarrow \CPR(F/N')$.
\end{customthm}


\begin{customthm}{\ref{th:nonCP_CP}}
$\CPR(F/N) \not\Rightarrow \CPR(F/N')$.
\end{customthm}



Our results give the following reducibility diagram
for the decision problems in groups of the form $F/N^{(d)}$:
$$
\xymatrix{
\WP(F/N) \ar@{<->}[r]\ar@{<->}[rd] & \WP(F/N') \ar@{<->}[r]\ar@{<->}[rd]\ar@{<->}[d] & \WP(F/N'') \ar@{<->}[d]\ar@{<->}[r] & \ldots\\
\PP(F/N) \ar@{->}[u]\ar@{->}[r]\ar@{<->}[rd] & \PP(F/N') \ar@{<->}[r]\ar@{<->}[rd] & \PP(F/N'') \ar@{<->}[r]& \ldots\\
\CPR(F/N) \ar@/^2pc/[uu] & \CPR(F/N') \ar@{->}[u]\ar@{->}[r] & \CPR(F/N'') \ar@{<->}[u] \ar@{<->}[r] & \ldots\\
}
$$
In particular, the following theorem holds.

\begin{customco}{\ref{co:second_step}}
For every recursively enumerable $N\unlhd F$ the following holds.
\begin{itemize}
\item[(a)]
$\PP(F/N') \Leftrightarrow \PP(F/N'')$.
\item[(b)]
$\CPR(F/N'') \Leftrightarrow \CPR(F/N''')$.
\item[(c)]
$\WP(F/N'') \Leftrightarrow \PP(F/N'') \Leftrightarrow \CPR(F/N'')$.
\qed
\end{itemize}
\end{customco}

Furthermore, most of the reductions are polynomial time computable.
Denote by $\P$ the class of decision problems decidable in polynomial time.

\begin{theorem*}
Suppose that $\WP(F/N) \in\P$. Then the problems
$\WP(F/N^{(d)})$, $\PP(F/N^{(d)})$, and $\CPR(F/N^{(d)})$ are in $\P$ for every $d\ge2$.
Moreover, each of those problems has a unique polynomial bound that does not depend on $d$.
\qed
\end{theorem*}

Finally, in Section \ref{se:combin_problems} we consider two combinatorial
problems for groups $F/N'$: subset sum problem $\SSP(F/N')$
and acyclic graph problem $\AGP(F/N')$. The main results of that sections are:

\begin{customthm}{\ref{th:SSP_P}}
$\SSP(F/N') \in \P$ if and only if $\WP(F/N) \in \P$ and
either $N=\{1\}$ or $[F:N]<\infty$.
\qed
\end{customthm}

\begin{customthm}{\ref{th:SSP_NPcomp}}
If $\WP(F/N) \in\P$, $N\ne \{1\}$, and $[F:N]=\infty$, then $\SSP(F/N')$
and $\AGP(F/N')$ are $\NP$-complete.
\qed
\end{customthm}

\begin{customco}{\ref{co:AGP_P}}
$\AGP(F/N') \in \P$ if and only if $\WP(F/N) \in \P$ and
either $N=\{1\}$ or $[F:N]<\infty$.
\qed
\end{customco}

\section{Preliminaries: $X$-digraphs}
\label{se:X_digraphs}
Let $X$ be a set (called an \emph{alphabet}) and $F=F(X)$
the free group on $X$. By $X^-$ we denote the set of formal inverses
of elements in $X$ and put $X^\pm = X\cup X^-$.
An $X$-\emph{labeled directed graph} $\Gamma$ (or an
$X$-\emph{digraph}) is a pair of sets $(V,E)$ where
the set $V$ is called the {\em vertex set} and the set
$E \subseteq V \times V \times X^\pm$ is called the {\em edge set}.
An element $e = (v_1,v_2,x) \in E$ designates an edge with
the {\em origin} $v_1$ (also denoted by $\alpha(e)$), the {\em terminus} $v_2$
(also denoted by $\omega(e)$), labeled with $x$ (also denoted by $\mu(e)$).
We often use notation $v_1\stackrel{x}{\rightarrow} v_2$
to denote the edge $(v_1,v_2,x)$. A {\em path} in $\Gamma$
is a sequence of edges $p=e_1,\ldots,e_k$ satisfying
$\omega(e_i)=\alpha(e_{i+1})$ for every $i=1,\ldots,k-1$.
The {\em origin} $\alpha(p)$ of $p$ is the vertex $\alpha(e_1)$,
the {\em terminus} $\omega(p)$ is the vertex $\omega(e_k)$,
and the {\em label} $\mu(p)$ of $p$ is the word $\mu(e_1)\ldots \mu(e_k)$.
We say that an $X$-digraph $\Gamma$ is:
\begin{itemize}
    \item
{\em rooted} if it has a special vertex, called the root;
    \item
{\em folded} (or deterministic) if for every
$v \in V$ and $x\in X$ there exists at most one edge with the origin
$v$ labeled with $x$;
    \item
{\em $X$-complete} (or simply complete) if for every $v_1\in V$ and $x\in X^\pm$
there exists an edge $v_1\stackrel{x}{\rightarrow} v_2$;
    \item
{\em inverse} if with every edge
$e=g_1\stackrel{x}{\rightarrow} g_2$ the graph $\Gamma$ also contains the {\em inverse edge}
$g_2\stackrel{x^{-1}}{\rightarrow} g_1$, denoted by $e^{-1}$.
\end{itemize}
All $X$-digraphs in this paper are connected.
A {\em morphism} of two rooted $X$-digraphs is a graph morphism which maps
the root to the root and preserves labels. For more information on $X$-digraphs we refer to
\cite{Stallings:1983,Kapovich_Miasnikov:2002}.

\begin{example}
Let $F=F(X)$ and $H\le F$.
The \emph{Schreier graph} of the subgroup $H$, denoted by $\Sch(X;H)$,
is an $X$-digraph $(V,E)$, where $V$ is the set of right cosets
$$V = \{Hg\mid g\in F\}$$
and
$$
  E=\{ Hg\stackrel{x}{\rightarrow} Hgx \mid g\in F,\ x\in X^{\pm}\}.
$$
By definition, $\Sch(X;H)$ is a folded complete inverse $X$-digraph.
We always assume that $H$ is the root of $\Sch(X;H)$.
A special case of the Schreier graph is when $H\unlhd F$,
called a \emph{Cayley graph} of the group $F/H$ denoted by $\Cay(X;H)$.
\qed
\end{example}

Let $\Gamma=(V,E)$ be an inverse $X$-digraph. The set of edges $E$ can be split
into a disjoint union $E = E^+ \sqcup E^-$, where
$$
E^+ = \{e\in E\mid \mu(e)\in X\}
$$
is called the set of \emph{positive edges}, and
$$
E^- = \{e\in E\mid \mu(e)\in X^-\}.
$$
is called the set of \emph{negative edges}.
Clearly, $(E^+)^{-1} = E^-$ and $(E^-)^{-1} = E^+$.

The \emph{rank} $r(\Gamma)$ of an inverse $X$-digraph $\Gamma$ is defined as
$|E^+| - |T|,$ where $T$ is any spanning subtree of $\Gamma$.
The fundamental group $\pi_1(\Gamma)$ is the group of labels
of all cycles at the root; it is naturally a subgroup of
$F(X)$ of the rank $r(\Gamma)$, see \cite{Kapovich_Miasnikov:2002}.

\section{Preliminaries: Computational model and data representation}
\label{se:Computational_Model}

All computations are assumed to be performed on
a random access machine.
We use base $2$ positional number system in which presentations of integers
are converted into integers via the rule:
    $$(a_{k-1}\ldots a_3a_2a_1a_0)_2 = a_{k-1}2^{k-1}+\ldots +a_22^2+a_12+a_0,$$
where we assume that $a_{k-1}=1$. The number $k$ is called the {\em bit-length} of the presentation.

Let $G$ be a group generated by a finite set $X = \{x_1,\ldots,x_n\}$.
We formally encode the word problem for $G$ as a subset of $\{0,1\}^\ast$
as follows. We first encode elements of the set $X^\pm = \{x_1^\pm,\ldots,x_n^\pm\}$
by unique bit-strings of length $\lceil \log_2 n\rceil+1$.
The code for a word $w=w(X_r^\pm)$ is a concatenation of codes for letters and, formally:
$$
\WP(F/N) = \{\code(w) \mid w\in N\}.
$$
Thus, the bit-length of the representation for a word $w \in F$ is:
    $$|\code(w)| = |w| (\lceil \log_2 n\rceil+1).$$
We encode the power and conjugacy problems in a similar fashion.
For both of these problems instances are pairs of words and the encoding
can be done by introducing a new letter ``,'' into the alphabet $X^{\pm}$.

\subsection{Quasi-linear time complexity}

An algorithm is said to run in {\em quasi-linear time} if its time complexity function
is $O(n \log^k n)$ for some constant $k\in \MN$.
We use notation $\tO(n)$ to denote quasi-linear time complexity.
Quasi-linear time algorithms are also $o(n^{1+\varepsilon})$
for every $\varepsilon > 0$, and thus run faster than any polynomial
in $n$ with exponent strictly greater than $1$.
See \cite{Naik-Regan-Sivakumar:1995} for more information on quasi-linear time complexity theory.
Similarly, one can define quasi-quadratic $\tO(n^2)$, quasi-cubic $\tO(n^3)$ time complexity
as $O(n^2\log^k n)$, $O(n^3\log^k n)$, etc.

\section{Flows on inverse $X$-digraphs}
\label{se:flows}

Let $\Gamma = (V,E)$ be an inverse $X$-digraph. We say that a function $f: E \rightarrow \MZ$
is \emph{balanced} if:
\begin{itemize}
    \item[\bf(F1)]
$f(e) = - f(e^{-1})$ for any $e \in E$.
\end{itemize}
All functions in this paper are balanced.
A function $f: E\to\MZ$ defines the function $\CN_f:V\to\MZ$:
$$
\CN_f(v) = \sum_{\alpha(e) = v} f(e),
$$
called the \emph{net-flow} function of $f$.
We say that $f$ is a \emph{flow} if it satisfies the conditions (F1), (F2), and (F3).
\begin{itemize}
    \item[\bf(F2)]
$f$ has a finite support $\supp(f) = \{e \in E \mid f(e) \neq 0\}$.
    \item[\bf(F3)]
Either $\CN_f(v) = 0$ for every $v \in V$ in which case we say that $f$ is a \emph{circulation},
or there exist $s,t\in V$ such that $\CN_f(v) = 0$ for all $v \in V\setminus\{s,t\}$,
and $\CN_f(s)  = 1$ and $\CN_f(t) = -1$ and we say that $f$ is a flow from the \emph{source} $s$ to the \emph{sink} $t$.
\end{itemize}
Define $\CF_\Gamma$ to be the set
of all balanced integral functions on $E$ with finite support:
$$
\CF_\Gamma = \{ f: E \rightarrow \MZ \}.
$$
For $f,g\in\CF_\Gamma$ define $f+g\in\CF_\Gamma$ as follows:
$$
(f+g)(e) = f(e)+g(e).
$$
Clearly, $(\CF_\Gamma,+)$ is an abelian group.
The function $\|\cdot\|:\CF_\Gamma\to\MZ$ defined by:
$$
\|\pi\| = \sum_{e\in E^+} |\pi(e)|
$$
is called a \emph{norm} on $\CF_\Gamma$.
It is easy to see that every $X$-digraph morphism
$\varphi:\Gamma\to \Delta$ induces a homomorphism of abelian groups $\rho_\varphi:\CF_\Gamma\to\CF_\Delta$ defined as follows:
$$
\rho_\varphi(f)(e') = \sum_{\varphi(e)=e'} f(e),
$$
for $f\in\CF_\Gamma$ and $e'\in E(\Delta)$. Clearly, $\|\pi\| \ge \|\rho_\varphi(\pi)\|$
for every $\pi\in\CF_\Ga$.

\subsection{Flows defined by words}

Let $\Gamma = (V,E)$ be an rooted folded complete inverse $X$-digraph and
$w=x_{i_1}^{\varepsilon_1} \ldots x_{i_k}^{\varepsilon_k}\in F(X)$.
The word $w$ defines a unique path $p_w$ in $\Gamma$:
$$
v_0 \stackrel{x_{i_1}^{\varepsilon_1}}{\to} v_1 \stackrel{x_{i_2}^{\varepsilon_2}}{\to} v_2
\stackrel{x_{i_3}^{\varepsilon_3}}{\to} \ldots
\stackrel{x_{i_k}^{\varepsilon_k}}{\to} v_k
$$
where $v_0$ is the root of $\Gamma$,
and a function $\pi_w^\Ga:E\to \MZ$ which associates to an edge $e$
the number of times $e$ is traversed minus the number of times $e^{-1}$
is traversed by $p_w$. It is easy to check that $\pi_w^\Ga$ is a flow in $\Gamma$.
We call $\pi_w^\Ga$ the \emph{flow} of $w$ in $\Gamma$.

\begin{figure}[t]
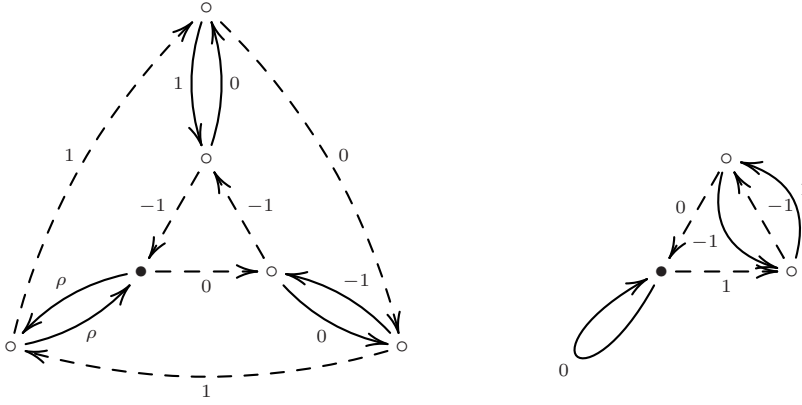

\centerline{
\xygraph{
!{<0cm,0cm>;<0cm,1cm>:<-1cm,0cm>::}
!{(0,0);a(0)**{}?(1.0)}*+{\circ}="a"
!{(0,0);a(120)**{}?(1.0)}*+{\bullet}="b"
!{(0,0);a(240)**{}?(1.0)}*+{\circ}="c"
!{(0,0);a(0)**{}?(3)}*+{\circ}="d"
!{(0,0);a(120)**{}?(3)}*+{\circ}="e"
!{(0,0);a(240)**{}?(3)}*+{\circ}="f"
"a":@[black]@[|(2)]@{-->}"b"_{-1}
"b":@[black]@[|(2)]@{-->}"c"_{0}
"c":@[black]@[|(2)]@{-->}"a"_{-1}
"f":@[black]@[|(2)]@/^0.4cm/@{-->}"e"^{1}
"d":@[black]@[|(2)]@/^0.4cm/@{-->}"f"^{0}
"e":@[black]@[|(2)]@/^0.4cm/@{-->}"d"^{1}
"a":@[black]@[|(2)]@/_0.2cm/"d"_{0}
"d":@[black]@[|(2)]@/_0.2cm/"a"_{1}
"c":@[black]@[|(2)]@/_0.2cm/"f"_{0}
"f":@[black]@[|(2)]@/_0.2cm/"c"_{-1}
"e":@[black]@[|(2)]@/_0.2cm/"b"_{\rho}
"b":@[black]@[|(2)]@/_0.2cm/"e"_{\rho}
}
\ \ \ \ \ \ \ \ \ \ \ \ \ \ \
\xygraph{
!{<0cm,0cm>;<0cm,1cm>:<-1cm,0cm>::}
!{(0,0);a(0)**{}?(1.0)}*+{\circ}="a"
!{(0,0);a(120)**{}?(1.0)}*+{\bullet}="b"
!{(0,0);a(240)**{}?(1.0)}*+{\circ}="c"
"a":@[black]@[|(2)]@{-->}"b"_{0}
"b":@[black]@[|(2)]@{-->}"c"_{1}
"c":@[black]@[|(2)]@{-->}"a"_{-1}
"b" :@`{"b"+(-2,1),"b"+(-1,2)}@[black]@[|(2)]"b"^{0}
"c":@[black]@[|(2)]@/_0.5cm/"a"_{1}
"a":@[black]@[|(2)]@/_0.5cm/"c"_{-1}
}
}
\caption{\label{fi:Cayley_S3} The Cayley graph of $S_3 = \gp{a,b}$ with $|a|=3$ and $|b|=2$
and the Schreier graph of $H=\gp{b}\le S_3$. The straight edges correspond to $b$ and dashed ones to $a$.
The values of $\pi_w$ for $w=[a^2,b]$ are shown on the edges.}
\end{figure}


\begin{lemma}[See {{\cite[Lemma 2.5]{Miasnikov_Romankov_Ushakov_Vershik:2010}}}]
For any flow $\pi:E(\Gamma)\to\MZ$ there exists $w\in F(X)$ satisfying $\pi=\pi_w^\Ga$.
\qed
\end{lemma}

In general, if $\Gamma$ is not complete, then some words can not be traced in $\Gamma$.
Suppose that a reduced nontrivial word $w$ can be traced in $\Gamma$.
The set of edges traversed by $w$ in
$\Gamma$ forms a connected $X$-digraph called the \emph{support graph}
of $w$ in $\Gamma$.

\begin{lemma}\label{le:trivial_flow_length}
Let $\Gamma$ be a rooted folded inverse $X$-digraph and $m$ the length
of a shortest cycle in $\Gamma$ (not necessarily at the root). Suppose that a reduced nontrivial
word $w$ can be traced in $\Gamma$ and $\pi_w^\Ga = 0$. Then $|w| \ge 3m$.
\end{lemma}

\begin{proof}
It follows from our assumption $\pi_w = 0$ that the path $p_w$ is a cycle in $\Gamma$.
Let $\Delta$ be the support graph of $w$ in $\Gamma$.
The rank of $\Delta$ can not be $0$ ($w$ is not reduced in this case)
and can not be $1$ (either $w$ is not reduced or $\pi_w\ne 0$).
Therefore, the rank of $\Delta$ is at least $2$.
Each edge of $\Delta$ is traversed by $w$ at least twice.
Hence, it is sufficient to prove that $2|E(\Delta)| \ge 3m$.
Let $\Delta'$ be a minimal subgraph of $\Delta$ of rank exactly $2$.
There are exactly two distinct configurations possible for $\Delta'$,
shown in Figure \ref{fi:support_min}.

\begin{figure}[htbp]
\centerline{ \includegraphics[scale=1]{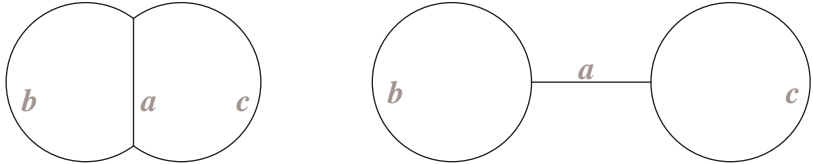} }
\caption{\label{fi:support_min} Two configurations for support graphs in Lemma \ref{le:trivial_flow_length}.}
\end{figure}

Let $a,b,c$ be the lengths of arcs as shown in the figure.
Since, the length of a shortest cycle in $\Gamma$ is $m$, we get the following bounds for our cases:
$$
\begin{array}{llll}
\left\{
\begin{array}{lll}
a+b \ge m, \\
a+c \ge m, \\
b+c \ge m, \\
\end{array}
\right.
&&&
\left\{
\begin{array}{lll}
b \ge m, \\
c \ge m. \\
\end{array}
\right.
\end{array}
$$
In both cases we have $2(a+b+c) \ge 3m$ which proves that $2|E(\Delta)| \ge 3m$.
Thus, $|w| \ge 3m$.
\end{proof}

\subsection{Flows on Schreier graphs}
In this section we study properties of flows on Schreier graphs.
The following lemma is the most important tool in the study of groups of the type
$F/N'$ and is the foundation of the Magnus embedding discussed in Section \ref{se:Magnus_emb}.
It is a well-known result and can be proven using algebraic topology techniques.
Here we provide a proof using elementary properties of Stallings' graphs.

\begin{lemma}\label{le:pi_NN}
Let $H\le F$, $\Delta=\Sch(X;H)$, and $w\in F$. Then $\pi_w^\Delta=0$ if and only if $w\in[H,H]$.
\end{lemma}

\begin{proof}
``$\Leftarrow$'' If $w\in[H,H]$, then $\pi_w = 0$.

``$\Rightarrow$''
Assume that $\pi_w^\Delta=0$. Then $w\in H$.
Taking a spanning tree in $\Sch(X;H)$ we can choose a good (perhaps infinite) set of generators $Y$ for $H$
corresponding to the cycles defined by positive edges outside of the spanning tree.
The word $w$ can be (uniquely) expressed as a word $w=u(Y)$ in the generators $Y$.
Since $\pi_w^\Delta=0$ we should have $\sigma_y(u)=0$
(algebraic sum of powers of $y$'s in the expression for $u$ is $0$)
for every $y\in Y$.
This means that $u$ can be expressed as a product of commutators of elements from $\gp{Y}=H$.
Hence $w\in[H,H]$.
\end{proof}

\begin{corollary}
$w=1$ in $F/N'$ if and only if $\pi_w^\Ga=0$ where $\Ga=\Cay(X;N)$.
\qed
\end{corollary}

\begin{corollary}
Let $H\le F$ and $m$ is the length of a shortest nonempty word in $H$.
Then the length of a shortest nonempty word in $[H,H]$ is at least $3m$.
\end{corollary}

\begin{proof}
By Lemma \ref{le:pi_NN} if $w\in[H,H]\setminus \{\varepsilon\}$, then $\pi_w^\Delta=0$
in $\Delta=\Sch(X;H)$.
By Lemma \ref{le:trivial_flow_length}, $|w|\ge 3m$.
\end{proof}

\begin{lemma}\label{le:pi_Sch}
Let $N\unlhd F$, $w\in F$, and $\Delta=\Sch(X;\gp{N,w})$. If $\pi_w^\Delta=0$, then $w\in N$.
\end{lemma}

\begin{proof}
If $\pi_w^\Delta=0$ then, by Lemma \ref{le:pi_NN}, $w\in[\gp{N,w},\gp{N,w}]$.
Hence, $w$ can be expressed as a product of commutators over $\gp{N,w}$.
That expresses $w$ as a product of elements from $N$ and $w$'s with the trivial
algebraic sum of powers for $w$. That product belongs to $N$ because $N$ is normal in $F$.
\end{proof}

\begin{corollary}\label{co:pi_Sch}
Let $N\unlhd F$, $w\in F\setminus N$, and $\Delta=\Sch(X;\gp{N,w})$. Then $\pi_w^\Delta\ne 0$.
\end{corollary}

\subsection{Flows on Cayley graphs}

Let $X=\{x_1,\ldots,x_n\}$, $F=F(X)$, $N\unlhd F$, $G=F/N$, and $\Gamma = \Cay(X;N)$.
The group $G$ acts on its Cayley graph $\Gamma$ by shifts:
$$
e^g = g^{-1}h\stackrel{x}{\to}g^{-1}hx.
$$
for $g\in G$ and $e=h\stackrel{x}{\to}hx \in E$.
The following action of $\MZ G$ on $\CF_\Gamma$ turns the later into a $\MZ G$-module.
For $c_1 g_1+\ldots+c_k g_k \in\MZ G$ and $f\in \CF_\Gamma$ define $f' = (c_1 g_1+\ldots+c_k g_k)f$
on $e=g\stackrel{x}{\to}gx$ to be:
$$
f'(e) = c_1 f(e^{g_1}) +\ldots+ c_k f(e^{g_k}).
$$
Denote $\pi_{x_i}$ by $\pi_i$ for $i=1,\ldots,n$.
The next lemma is straightforward.

\begin{lemma}
$\CF_\Gamma$ is a free $\MZ G$-module of rank $n$ with a free basis $\{\pi_1,\ldots,\pi_n\}$.
In particular, every $\pi\in\CF_\Ga$ can be uniquely expressed as a $\MZ G$ linear combination
of $\pi_1,\ldots,\pi_n$.
\qed
\end{lemma}

The set of circulation $\CC_\Gamma$ in $\CF_\Gamma$ is closed under addition and the scalar
$\MZ G$-multiplication and hence $\CC_\Gamma$ is a $\MZ G$-submodule of $\CF_\Gamma$.

\begin{lemma}
If $G$ is finitely presented, then $\CC_\Gamma$ is a finitely generated $\MZ G$-module.
\end{lemma}

\begin{proof}
If $N=\ncl_F(r_1,\ldots,r_k)$, then $\CC_\Gamma = \gp{\pi_{r_1},\ldots,\pi_{r_k}}$.
\end{proof}

There exists an algebraic way to define the net-flow function of $\pi\in \CF_\Gamma$.
Define a map $\CN:\CF_\Gamma\to \MZ G$ by:
$$
\pi=\alpha_1\pi_1+\ldots+\alpha_n \pi_n \stackrel{\CN}{\mapsto}
\sum_{i=1}^n\alpha_i(1-\ovx_i).
$$
If $\pi=\sum_{i=1}^n \sum_{g\in G} a_{g,i} g \pi_i$, then the coefficient for $g$ in $\CN(\pi)$ is
$$a_{g,1}+\ldots+a_{g,n} - a_{g\gamma(x_1)^{-1},1}-\ldots-a_{g\gamma(x_n)^{-1},n}$$
which is exactly the value of the net flow at $g$ defined by $\pi$.
Hence, $\CN(\pi)$ is a description of $\CN_\pi$ as an element of $\MZ G$.
The next propositions are obvious.

\begin{proposition}
$\CN:\CF_\Ga\to\MZ G$ is a $\MZ G$-module homomorphism.
\qed
\end{proposition}

\begin{proposition}\label{pr:t_flow}
For any $\pi\in \CF_\Gamma$ the following holds.
\begin{itemize}
\item[(a)]
$\pi\in \CC_\Gamma$ if and only if $\CN(\pi)=0$.
\item[(b)]
$\pi$ is a flow on $\Ga$ if and only if $\CN(\pi)=1-g$
for some $g\in G$.
\qed
\end{itemize}
\end{proposition}

Consider the \emph{augmentation map} $\epsilon:\MZ G\to \MZ$, defined by:
$$
\sum_{g\in G} \alpha_g g \quad\stackrel{\epsilon}{\mapsto} \quad\sum_{g\in G} \alpha_g.
$$
Recall that $\epsilon$ is a ring homomorphism and therefore a homomorphism of $\MZ G$-modules once $\MZ$ is assumed to be a $\MZ G$-module with trivial $G$ action.
\begin{lemma}\label{le:CN_nu}
The sequence $\CF_\Ga \stackrel{\CN}{\to} \MZ G \stackrel{\epsilon}{\to} \MZ$ is an exact sequence of $\MZ G$-modules.
\end{lemma}

\begin{proof}
The inclusion $\im(\CN)\subseteq \ker(\epsilon)$ can be easily shown by induction on $\|\pi\|$.
To show the opposite inclusion pick an arbitrary $\sum \alpha_g g \in\ker(\epsilon)$.
Then
$$
\sum \alpha_g g = \sum \alpha_g g - \sum \alpha_g = \sum \alpha_g (g-1).
$$
If $G\ni g=x_{i_1}^{\varepsilon_1} \ldots x_{i_k}^{\varepsilon_k}$,
then $g-1$ can be written as follows:
$$
(x_{i_1}^{\varepsilon_1} \ldots x_{i_k}^{\varepsilon_k}-x_{i_1}^{\varepsilon_i} \ldots x_{i_{k-1}}^{\varepsilon_{k-1}}) + (x_{i_1}^{\varepsilon_i} \ldots x_{i_{k-1}}^{\varepsilon_{k-1}}-x_{i_1}^{\varepsilon_i} \ldots x_{i_{k-2}}^{\varepsilon_{k-2}}) + \ldots + (x_{i_1}^{\varepsilon_1}-1)
$$
$$
=\sum_{j=0}^{k-1} x_{i_1}^{\varepsilon_1} \ldots x_{i_j}^{\varepsilon_j} (x_{i_{j+1}}^{\varepsilon_{j+1}}-1).
$$
Therefore, $-\sum \alpha_g g = \sum \alpha_g (1-g)$ is a linear
$\MZ G$-combination of the elements of the form $(1-\ovx_i)$
and hence belongs to $\im(\CN)$. Thus, $\sum \alpha_g g \in \im(\CN)$.
\end{proof}

\section{Magnus embedding}
\label{se:Magnus_emb}

Let $X = \{x_1,\ldots,x_n\}$, $F=F(X)$, $N\unlhd F$, and $\Gamma=\Cay(X;N)$.
In this section we study relations between groups $G=F/N$ and $F/N'$.
It is easy to check that the set of matrices:
$$
M(X;N) =
\Set{\left(
\begin{array}{cc}
g & \pi\\
0 & 1
\end{array}
\right)
}{g\in G,~\pi\in \CF_\Gamma}
$$
forms a group with respect to the matrix multiplication which can be easily recognized as the wreath product $\MZ^n \rwr G$.

Let $\overline{\phantom{\gamma}}:F\to F/N$ be the canonical epimorphism.
Define a homomorphism $\mu:F\to M(X;N)$ by:
\begin{equation}\label{eq:MagnusEmb}
x_i \stackrel{\mu}{\mapsto}
\left(
\begin{array}{cc}
\ovx_i & \pi_i\\
0 & 1
\end{array}
\right),
\ \ \
x_i^{-1} \stackrel{\mu}{\mapsto}
\left(
\begin{array}{cc}
\ovx_i^{-1} & -\ovx_i^{-1} \pi_i\\
0 & 1
\end{array}
\right).
\end{equation}
It is easy to check by induction on $|w|$ that:
\begin{equation}\label{eq:elem_M}
\mu(w) =
\left(
\begin{array}{cc}
\ovw & \pi_w\\
0 & 1
\end{array}
\right).
\end{equation}

\begin{proposition}\label{pr:MagnusEmb}
For any $w\in F$ if $\pi_w^\Ga=0$ then $\ovw=1$.
\end{proposition}

\begin{proof}
Assume that $\ovw\ne 1$ in $F/N$. Tracing $w$ in $\Ga$
we obtain a path $p_w$ from $1$ to $wN\ne 1$. The path $p_w$ is not a circuit
and the corresponding flow is not a circulation, i.e.
$$
\CN_{\pi_w}(1)=\sum_{\alpha(e) = 1} \pi_w^\Ga(e) = 1.
$$
Therefore, $\pi_w^\Ga(e)\ne 0$ for some edge $e$ adjacent to $1$. Thus, $\pi_w^\Ga\ne 0$.
\end{proof}

Now note that for every $w\in F$:
\begin{align*}
w\in \ker(\varphi) &\Leftrightarrow \pi_w^\Ga=0 \mbox{ and }\ovw=1\\
&\Leftrightarrow \pi_w^\Ga=0 \mbox{ (by Proposition \ref{pr:MagnusEmb})}\\
&\Leftrightarrow w\in N' \mbox{ (by Lemma \ref{le:pi_NN})},
\end{align*}
which proves the following theorem.

\begin{theorem}[See \cite{Magnus:1939}]\label{th:MagnusEmb}
Let $F = F(x_1,\ldots,x_n)$, $N\unlhd F$, and $\overline{\phantom{\gamma}}:F\to F/N$ be the canonical epimorphism.
The homomorphism $\mu:F\to M(X;N)$ defined by
$$
x \stackrel{\mu}{\mapsto}
\left(
\begin{array}{cc}
\ovx & \pi_i\\
0 & 1
\end{array}
\right)
$$
satisfies $\ker(\mu) = N'$. Therefore, $F/N'\simeq \mu(F) \le M(X;N)$.
The induced embedding $F/N' \to M(X;N)$ is called the \emph{Magnus embedding}.
\qed
\end{theorem}

\subsection{Properties of Magnus embedding}
\label{se:PropMagnusEmbedding}

The following proposition was proved in \cite{Remeslennikov-Sokolov:1970}
using Fox derivatives. Let $g\in G$, $\pi=\sum_{i=1}^n \alpha_i \pi_i \in\CF_\Ga$, and
$$
A=
\left(
\begin{array}{cc}
g & \sum_{i=1}^n \alpha_i \pi_i\\
0 & 1
\end{array}
\right).
$$

\begin{proposition}[See {{\cite[Theorem 2]{Remeslennikov-Sokolov:1970}}}]
The following holds.
\begin{itemize}
\item[(a)]
$A\in\mu(F)$ if and only if $\sum_{i=1}^n \alpha_i(1-\ovx_i)=1-g$.
\item[(b)]
$A\in\mu(N)$ if and only if $\sum_{i=1}^n \alpha_i(1-\ovx_i)=0$.
\end{itemize}
\end{proposition}

\begin{proof}
Follows from (\ref{eq:elem_M}) and Proposition \ref{pr:t_flow}.
\end{proof}

For a nontrivial $g\in G$ define a map $\tau_g:\CF_\Ga\to \CF_\Ga$ by:
$$
\pi \quad\stackrel{\tau_g}{\mapsto} \quad(1-g)\pi.
$$
Denote $\Sch(X;\gp{N,g})$ by $\Delta$.
The natural projection $\varphi:\Gamma \to \Delta$ induces
an abelian group homomorphism $\rho_g:\CF_\Gamma\to\CF_\Delta$.
Properties of the functions $\tau_g$ and $\rho_g$ are very important in the study of
the conjugacy problem in $F/N'$.

\begin{lemma}\label{le:tau_im}
The sequence $\CF_\Gamma \stackrel{\tau_g}{\to} \CF_\Gamma \stackrel{\rho_g}{\to} \CF_\Delta$ is an exact sequence of abelian groups.
\end{lemma}

\begin{proof}
The image of $\tau_g$ is a subgroup of
$\CF_\Gamma$ generated by the elements $(1-g)h\pi_i$ for $h\in G$ and $i=1,\ldots,n$.
Clearly, $\rho_g((1-g)h\pi_i) = 0$.

Conversely, assume that $\pi'=\rho_g(\pi) = 0$. Let $H=\gp{N,g}$.
Consider any edge $e'=Hh\stackrel{x_i}{\to}Hhx_i$ in $\Delta$. By definition of $\rho_g$
we get
$$
0=\pi'(e') = \sum_{\varphi(e)=e'} \pi(e) = \sum_{g^j} \pi(Ng^jh\stackrel{x_i}{\to}Ng^jhx_i),
$$
where $g^j$ ranges over all distinct powers of $g$. It is easy to see that such $\pi$
is an integral linear combination of the elements $(1-g)h\pi_i$ and, hence, belongs to $\im(\tau_g)$.
\end{proof}

\begin{lemma}[Cf. {{\cite[Lemma 4]{Gupta:1982}}}]\label{le:tau_ker}
$\ker(\tau_g)$ is not trivial if and only if $|g|=k<\infty$,
in which case it is an abelian subgroup of $\CF_\Ga$:
$$
\ker(\tau_g) = \gp{(1+g+\ldots+g^{k-1})h\pi_i \mid h\in G \mbox{ and }i=1,\ldots,n}.
$$
\end{lemma}

\begin{proof}
Pick any $\pi\in\CF_\Ga$ such that $(1-g)\pi=0$.
Then $g^j\pi=\pi$ for every $j\in \MZ$, which can not happen if $|g|=\infty$
since $\pi$ has finite support. Assume that $|g|=k<\infty$.
It is straightforward to check that $(1+g+\ldots+g^{k-1})h\pi_i \in \ker(\tau_g)$. On the other hand if
$g^j\pi=\pi$ for every $j\in \MZ$, then the coefficients $\alpha_{h,i}$ are constant on right
$\gp{g}$-cosets and hence are linear combinations of the generators.
\end{proof}

\begin{lemma}\label{le:circulation_1_t}
Let $\pi\in \CF_\Ga$ and $1\ne g\in G$.
If $(1-g)\pi \in \CC_\Gamma$, then there exists $\pi^\ast\in\CC_\Gamma$
satisfying $(1-g)\pi = (1-g)\pi^\ast$.
\end{lemma}

\begin{proof}
If $(1-g)\pi\in\CC_\Gamma$, then
$$
\CN((1-g)\pi) =\CN(\pi)-g\CN(\pi) = 0.
$$
Therefore, $\CN(\pi) = g^j\CN(\pi)$ for every $j\in\MZ$.

{\sc Case-I.} Assume that $|g|=\infty$.
Since $\CN(\pi)$ has finite support, it must be the case that $\CN(\pi)=0$.
Thus, $\pi \in \CC_\Gamma$.

{\sc Case-II.} Assume that $|g|=k<\infty$. Then:
$$
\CN(\pi) = g\CN(\pi) = \ldots = g^{k-1}\CN(\pi),
$$
i.e., $\CN(\pi)$ is constant on right $\gp{g}$-cosets. Hence,
$$
\CN(\pi) = (1+g+\ldots+g^{k-1}) N',\quad \mbox{ for some } N' \in \MZ G.
$$
By Lemma \ref{le:CN_nu}, we have:
$$
0 = \epsilon(\CN(\pi)) = \epsilon(1+g+\ldots+g^{k-1}) \epsilon(N') = k \epsilon(N') \quad \mbox{ in }\MZ,
$$
which implies that $\epsilon(N') = 0$. Hence $N'\in\ker(\epsilon)$ and by Lemma \ref{le:CN_nu}:
$$
N' = \sum \alpha_i (1-\ovx_i) \quad \mbox{ for some } \alpha_i\in\MZ G.
$$
Put $\pi' = \alpha_1\pi_1+\ldots+\alpha_n\pi_n \in\CF_\Ga$. Notice that $N'=\CN(\pi')$ and:
$$
\CN(\pi) = (1+g+\ldots+g^{k-1})\CN(\pi').
$$
Define $\pi^\ast = \pi-(1+g+\ldots+g^{k-1})\pi'$ and observe that:
$$
\CN(\pi^\ast) = 0 \mbox{ and } (1-g)\pi=(1-g)\pi^\ast,
$$
i.e., $\pi^\ast$ is a required element.
\end{proof}

\section{Algorithmic properties of groups $F/N^{(d)}$}

In this section we study relations between algorithmic problems in groups
$\{F/N^{(d)}\}_{d=0}^\infty$, where $N^{(d)}$ denotes the $d$th derived
subgroup of $N$.

\subsection{Word problem}

Here we review the relations between the word problems
in groups $\{F/N^{(d)}\}$.

\begin{proposition}\label{pr:WP_WP}
$\WP(F/N) \Rightarrow \WP(F/N')$.
\end{proposition}

\begin{proof}
Let $\Ga=\Cay(X;N)$.
By Lemma \ref{le:pi_NN},
$w = x_{i_1}^{\varepsilon_1} \ldots x_{i_k}^{\varepsilon_k}$
represents the identity in $F/N'$ if and only if $\pi_w^\Ga = 0$.
To describe the function $\pi_w^\Ga$ one needs to distinguish edges in $\Ga$
traversed by $w$:
$$
\varepsilon \stackrel{x_{i_1}^{\varepsilon_1}}{\to} w_1 \stackrel{x_{i_2}^{\varepsilon_2}}{\to} w_2 \stackrel{x_{i_2}^{\varepsilon_2}}{\to}  \ldots
\stackrel{x_{i_k}^{\varepsilon_k}}{\to}
x_{i_1}^{\varepsilon_1} \ldots x_{i_k}^{\varepsilon_k},
$$
which can be done because by assumption $\WP(F/N)$ is decidable.
\end{proof}

\begin{proposition}\label{pr:shortest identity}
Let $w \in F \setminus \{\varepsilon\}$.
If $w=1$ in $F/N^{(d)}$, then $|w|\ge 3^d$.
\end{proposition}

\begin{proof}
Easy induction on $d$ using Lemma \ref{le:trivial_flow_length}.
\end{proof}

\begin{corollary}
For any $N\unlhd F$ the sequence of groups $F/N^{(d)}$ converges to $F$ in Gromov-Hausdorff topology.
\qed
\end{corollary}

\begin{theorem}\label{th:WP_WP}
If $\WP(F/N)\in\tO(T(n))$, then $\WP(F/N') \in\tO(T(n)n^2)$.
Furthermore, $\WP(F/N^{(d)})\in\tO(T(n)n^2)$ for every $d\in\MN$.
\end{theorem}

\begin{proof}
It requires $|w|^2$ calls of the $\WP(F/N)$ to construct the support graph for a given word $w$.
Given the support graph for $w$ it is straightforward to construct the flow for $w$ in $\Cay(F/N)$ as described in \cite[Section 2]{Ushakov:2014}.
By Proposition \ref{pr:shortest identity},
to solve $\WP(F/N^{(d)})$ one must iterate the procedure at most $\log_3|w|$ times.
\end{proof}

The converse to Proposition \ref{pr:WP_WP} also holds.
It was first proven in \cite{Anokhin:1997} using Bronstein monotonocity theorem.
Here we use Auslander-Lyndon theorem which gives the most straightforward algorithm.

\begin{theorem}
Assume that $N$ is a recursively enumerable normal subgroup of $F$ and $N'$ is recursive, then $N$ is recursive.
\end{theorem}

\begin{proof}
The statement is obvious for abelian $F$ or $N=\{1\}$.
Assume that $F$ is not abelian and $N$ is not trivial.
Then $N$ has rank at least $2$.
By \cite[Theorem 1]{Auslander-Lyndon:1955}, for any $v\in F\setminus N$
there exists $w\in N$ such that $[v,w]\notin N'$.
That gives a procedure for testing if $v\notin N$.
Thus, $N$ is recursive.
\end{proof}

\begin{corollary}
Assume that $N$ is a recursively enumerable normal subgroup of $F$ and $\WP(F/N^{(d)})$ is decidable for some $d\in\MN$.
Then $\WP(F/N^{(d)})$ is decidable for every $d\in\MN$.
\qed
\end{corollary}

\subsection{Power problem}

In this section we use properties of Magnus embedding to prove that
the group $F/N'$ is torsion free and to solve the power problem in $F/N'$.
Let $\Ga=\Cay(X;N)$.

\begin{lemma}\label{le:norm_growth}
For every $w\notin N'$ and $k\in\MN$ we have $\|\pi_{w^k}^\Ga\|\ge k$.
\end{lemma}

\begin{proof}
Assume that $w\notin N'$ and consider two cases.\\
{\sc Case-I:} If $w\in N$, then $\|\pi_w^\Ga\|\ge 1$ and
$\|\pi_{w^k}^\Ga\| = \|k\pi_w^\Ga\|\ge k$ for every $k\in\MN$. \\
{\sc Case-II:} Assume that $w\notin N$ and denote $\Sch(X;\gp{N,w})$ by $\Delta$.
By Lemma \ref{le:pi_Sch},
$\|\pi_w^\Delta\|\ge 1$ and $\|\pi_{w^k}^\Delta\|=\|k\pi_{w}^\Delta\| \ge k$ for every $k\in\MN$.
The later implies that $\|\pi_{w^k}^\Ga\|\ge k$.
\end{proof}

\begin{theorem}\label{th:torsion_free}
For every $N\unlhd F$ the group $F/N'$ is torsion free.
\end{theorem}

\begin{proof}
By Lemma \ref{le:norm_growth},
if $w\notin N'$ and $k\in\MN$, then $\|\pi_{w^k}^\Ga\|\ge k$, i.e., $w^k\notin N'$.
\end{proof}

\begin{lemma}\label{le:bounded_powers}
Let $u,v\in F$ and $u\notin N'$. If $u^k=v$ in $F/N'$, then $k\le |v|$.
\end{lemma}

\begin{proof}
If $|v|<k$, then $\|\pi_v\|<k\le \|\pi_{u^k}\|$, which means that
$u^k\ne v$ in $F/N'$.
\end{proof}

\begin{theorem}\label{th:Magnus_PP}
If $\WP(F/N)$ is decidable, then $\PP(F/N')$ is decidable.
Furthermore, if $\WP(F/N)\in\tO(T(n))$, then $\PP(F/N^{(d)})\in\tO(T(L^2)L)$ for every $d\in\MN$, where $L=|u|+|v|$
is the size of the input.
\end{theorem}

\begin{proof}
By Lemma \ref{le:bounded_powers}, given $u,v\in F$ it is sufficient to check if $v=u^k$ in $F/N'$
for $k=-|v|,\ldots,|v|$ which reduces to $2|v|+1 = O(L)$
calls of the word problem in $F/N'$
for words $v^{-1} u^{-|v|},\ldots,v^{-1} u^{|v|}$ whose
lengths are bounded by $|v|+|u|\cdot |v| = O(L^2)$.
\end{proof}

\subsection{Conjugacy problem}

Matthews proved in \cite[Theorem B]{Matthews:1966} that:
$$
\CPR(M(X;N)) \Leftrightarrow
\left\{
\begin{array}{l}
\CPR(F/N),\\
\PP(F/N).
\end{array}
\right.
$$
The main result of this section states that restricting the conjugacy problem
from $M(X;N)$ to $F/N'$ gives a problem
equivalent to $\PP(F/N)$. In general, decidability of $\CPR(F/N)$ is irrelevant
to decidability of $\CPR(F/N')$.

The theorem below was first proved by Remeslennikov and Sokolov in
\cite{Remeslennikov-Sokolov:1970} for a torsion free group $F/N$
and by C. Gupta in \cite{Gupta:1982} for any finitely generated group $F/N$.

\begin{theorem}\label{le:RS_conj}
For any $u,v\in F$ the matrices
$$
\mu(u) =
\left(
\begin{array}{cc}
\ovu & \pi_u^\Ga\\
0 & 1
\end{array}
\right)
\mbox{ and }
\mu(v) =
\left(
\begin{array}{cc}
\ovv & \pi_v^\Ga\\
0 & 1
\end{array}
\right)
$$
are conjugate in $M(X;N)$ if and only if they are conjugate in $\mu(F)$.
\end{theorem}

\begin{proof}
Assume that for some $c\in F$ and $\pi\in \CF_\Ga$ the matrix
\begin{equation}\label{eq:conjug_M}
w=
\left(
\begin{array}{cc}
\ovc & \pi\\
0 & 1
\end{array}
\right)
\in M
\end{equation}
conjugates $\mu(u)$ into $\mu(v)$.
Then $\ovc^{-1}\ovu\ovc=\ovv$.

{\sc Case-I.} If $\ovu=\ovv=1$ in $F/N$, then it is easy to check that the matrix
$$
\left(
\begin{array}{cc}
\ovc & \pi_c^\Ga\\
0 & 1
\end{array}
\right)
\in \mu(F)
$$
is a conjugator for $\mu(u)$ and $\mu(v)$ as well.

{\sc Case-II.} Assume that $\ovu,\ovv\ne1$ in $F/N$.
Conjugating $\mu(u)$ by $\mu(c)$ we obtain an equivalent instance of
the conjugacy problem with the conjugator (\ref{eq:conjug_M}) having trivial
upper-left entry. Hence, from the beginning we may assume that the $\ovc=1$
and $\ovu=\ovv$. That gives the equality:
$$
\left(
\begin{array}{cc}
1 & -\pi\\
0 & 1
\end{array}
\right)
\left(
\begin{array}{cc}
\ovu & \pi_u^\Ga\\
0 & 1
\end{array}
\right)
\left(
\begin{array}{cc}
1 & \pi\\
0 & 1
\end{array}
\right)
=
\left(
\begin{array}{cc}
\ovv & \pi_v^\Ga\\
0 & 1
\end{array}
\right)
$$
which implies that $\pi_u^\Ga-\pi_v^\Ga=(1-\ovu)\pi$. Since $\pi_u^\Ga-\pi_v^\Ga\in \CC_\Gamma$,
by Lemma \ref{le:circulation_1_t}, there exists a circulation $\pi^\ast$ satisfying
$\pi_u^\Ga-\pi_v^\Ga=(1-\ovu)\pi^\ast$. The matrix
$$
\left(
\begin{array}{cc}
1 & \pi^\ast\\
0 & 1
\end{array}
\right)
\in \mu(F)
$$
is a required conjugator for $\mu(u)$ and $\mu(v)$.
\end{proof}

\begin{theorem}[Geometry of conjugacy problem]\label{th:Geometric_Magnus_conjugacy}
Let $N\unlhd F$, $u,v\in F$, and $\Delta=\Sch(X,\gp{N,u})$. Then
$u\sim v$ in $F/N'$ if and only if there exists $c\in F$ satisfying the conditions:
\begin{itemize}
\item[(a)]
$\pi_u^\Delta=\ovc\pi_v^\Delta$, i.e., $\pi_u$ can be obtained
by a $\ovc$-shift of $\pi_v$ in $\Delta$;
\item[(b)]
$\ovc^{-1} \ovu \ovc=\ovv$ in $F/N$.
\end{itemize}
\end{theorem}

\begin{proof}
By Theorem \ref{le:RS_conj}, $u\sim v$ in $F/N'$ if and only if $\mu(u)\sim\mu(v)$ in $M(X;N)$.
The conjugacy equation for $\mu(u)$ and $\mu(v)$ is:
$$
\left(
\begin{array}{cc}
\ovc^{-1} & -\ovc^{-1}\pi\\
0 & 1
\end{array}
\right)
\left(
\begin{array}{cc}
\ovu & \pi_u^\Ga\\
0 & 1
\end{array}
\right)
\left(
\begin{array}{cc}
\ovc & \pi\\
0 & 1
\end{array}
\right)
=
\left(
\begin{array}{cc}
\ovv & \pi_v^\Ga\\
0 & 1
\end{array}
\right)
$$
for some $c\in F$ and $\pi\in \CF_\Ga$, which is equivalent to the system:
\begin{equation}\label{eq:Magnus_conj_eq}
\left\{
\begin{array}{l}
\ovc^{-1}\ovu\ovc=\ovv,\\
\pi_u^\Ga-\ovc\pi_v^\Ga = (1-\ovu)\pi.
\end{array}
\right.
\end{equation}
By Lemma \ref{le:tau_im}, the equality $\pi_u^\Ga-\ovc\pi_v^\Ga = (1-\ovu)\pi$ holds
if and only if
$\pi_u^\Delta-\ovc\pi_v^\Delta = 0$.
\end{proof}

\begin{theorem}\label{th:Magnus_CP}
If $\PP(F/N)$ is decidable, then $\CPR(F/N')$ is decidable.
Furthermore, if $\PP(F/N)\in\tO(T(n))$, then $\CPR(F/N^{(d)})\in\tO(T(L^2)L^4)$ for every $d\in\MN$, where $L=|u|+|v|$
is the size of the input.
\end{theorem}

\begin{proof}
We may assume that $u\ne 1$ and $v\ne 1$ in $F/N'$.
If $u=1$ in $F/N$, then $\pi_u^\Delta = \pi_u^\Ga \ne 0$.
If $u\ne 1$ in $F/N$, then by Corollary \ref{co:pi_Sch}, we get the same:
$$\pi_u^\Delta\ne 0.$$
It shows that Case (2) in the proof of \cite[Theorem B]{Matthews:1966}
is impossible in $F/N'$ and allows to drop decidability of $\CPR(F/N)$.
The rest of the proof is essentially the same as that of
Case (3) in the proof of \cite[Theorem B]{Matthews:1966}.

Let $u=x_1\ldots x_s\in F$ and $v=y_1\ldots y_t\in F$, where $x_i,y_j\in X^\pm$.
For $i=1,\ldots,s$ and $j=1,\ldots,t$ define words:
$$
u_i = x_1\ldots x_i, \quad v_j = y_1\ldots y_j, \quad \mbox{and }
c_{i,j} = u_i v_j^{-1}.
$$
Fix $i$ such that $\pi_u^\Delta(u_{i-1} \to u_i) \ne 0$.
The set of solutions of the equation $\pi_u=\ovc\pi_v$ is a finite (perhaps trivial)
union of some cosets $\gp{u}c_{i,j}$ in $F/N$.
Since $\PP(F/N)$ is decidable we can directly check the equalities
$\pi_u^\Delta=\ovc_{i,j}\pi_v^\Delta$
and $\ovc_{i,j}^{-1}\ovu\ovc_{i,j}=\ovv$.
Hence, $\CPR(F/N')$ is decidable.

For every $c_{i,j}$ the described algorithm
for $\CPR(F/N')$ makes $O(L^2)$ calls of the subroutine for $\PP(F/N)$
with instances $(u,c_{i,j}v_ku_l)$ of size $O(L)$.
Hence overall complexity of testing $\pi_u=\ovc\pi_v$
is $\tO(T(L^2)L \cdot L^2\cdot L$. Thus the claimed complexity bound.
\end{proof}

Our next goal is to prove the converse of Theorem \ref{th:Magnus_CP}.

\begin{proposition}\label{pr:CP_PP}
Let $N\unlhd F$ and $u,v \in F\setminus N$ satisfy $[u,v]=1$ in $F/N$.
Then $v\in\gp{u}$ in $F/N$ if and only if $u\sim u[w,v]$ in $F/N'$ for every $w\in\gp{N,u}$.
\end{proposition}

\begin{proof}
Let $\Delta=\Sch(X;\gp{N,u})$.
Since $[u,v]=1$ in $F/N$, we have $v^{-1}\gp{N,u}v=\gp{N,u}$, i.e., $v$ normalizes $\gp{N,u}$.
Therefore, $\gp{N,u} \unlhd \gp{N,u,v}$ and the following holds.
\begin{align*}
v\in\gp{u} \mbox{ in }F/N & \Leftrightarrow v\in\gp{N,u} \mbox{ in }F\\
& \Leftrightarrow [w,v]\in\gp{N,u}'\le N, \quad \forall w\in \gp{N,u} \quad \mbox{(by \cite[Theorem 1]{Auslander-Lyndon:1955})}\\
& \Leftrightarrow \pi_{[w,v]}^\Delta=0 \mbox{ and } [w,v]=1 \mbox{ in }F/N, \quad \forall w\in \gp{N,u}\\
& \Leftrightarrow \pi_u^\Delta=\pi_{u[w,v]}^\Delta \mbox{ and } u=u[w,v] \mbox{ in }F/N, \quad \forall w\in \gp{N,u}.
\end{align*}
By Theorem \ref{th:Geometric_Magnus_conjugacy} the later holds if and only if
$u\sim u[w,v]$ in $F/N'$ for every $w\in \gp{N,u}$.
\end{proof}

\begin{theorem}[Cf. {{\cite[Lemma 1]{Anokhin:1997}}}]\label{th:CP_PP}
Assume that $N$ is recursively enumerable. Then
$\CPR(F/N') \Rightarrow \PP(F/N)$.
\end{theorem}

\begin{proof}
By assumption $\CPR(F/N')$ is decidable. Hence
$\WP(F/N')$ is decidable and $\WP(F/N)$ is decidable.
Consider an arbitrary instance $u,v\in F$ of $\PP(F/N)$.
Our goal is to decide if $v\in\gp{u}$ in $F/N$, or not.
\begin{itemize}
\item
If $v=1$ in $F/N$, then the answer is YES.
\item
If $u=1$ in $F/N$ and $v\ne1$, then the answer is NO.
\item
If $[u,v]\ne 1$ in $F/N$, then the answer is NO.
\end{itemize}
Hence, we may assume that $u\ne 1$, $v\ne 1$, and $[u,v]=1$ in $F/N$.

To test if $v\in\gp{u}$ in $F/N$ we
run a process that checks if $v=u^k$ in $F/N$ for some $k\in\MZ$.
To test if $v\notin\gp{u}$ in $F/N$ we enumerate all words $w\in\gp{N,u}$
and solve the conjugacy problem for words $u$ and $u[w,v]$ in $F/N'$.
By Proposition \ref{pr:CP_PP}, if $v\notin\gp{u}$
then a negative instance will be found eventually.
\end{proof}


\subsection{Relations among the algorithmic problems}

Theorems \ref{th:Magnus_PP}, \ref{th:Magnus_CP}, and \ref{th:CP_PP} give the following diagram of problem reducibility
for a finitely generated recursively presented group $F/N$:
$$
\xymatrix{
\WP(F/N) \ar@{<->}[r]\ar@{<->}[rd] & \WP(F/N') \ar@{<->}[r]\ar@{<->}[rd]\ar@{<->}[d] & \WP(F/N'') \ar@{<->}[d]\ar@{<->}[r] & \ldots\\
\PP(F/N) \ar@{->}[u]\ar@{->}[r]\ar@{<->}[rd] & \PP(F/N') \ar@{<->}[r]\ar@{<->}[rd] & \PP(F/N'') \ar@{<->}[r]& \ldots\\
\CPR(F/N) \ar@/^2pc/[uu] & \CPR(F/N') \ar@{->}[u]\ar@{->}[r] & \CPR(F/N'') \ar@{<->}[u] \ar@{<->}[r] & \ldots\\
}
$$
Below we deduce a few corollaries in the spirit of \cite[Problem 12.98]{Kourovka}.

\begin{corollary}\label{co:second_step}
For every recursively enumerable $N\unlhd F$ the following holds.
\begin{itemize}
\item[(a)]
$\PP(F/N')$ is decidable if and only if $\PP(F/N'')$ is decidable.
\item[(b)]
$\CPR(F/N'')$ is decidable if and only if $\CPR(F/N''')$ is decidable.
\item[(c)]
$\WP(F/N'')$ is decidable if and only if $\PP(F/N'')$ is decidable if and only if $\CPR(F/N'')$ is decidable.
\qed
\end{itemize}
\end{corollary}

Existence of a group $F/N$ with decidable $\CPR(F/N)$ and undecidable $\CPR(F/N')$
was shown by Anokhin in \cite{Anokhin:1997}
(his result is somewhat similar to Theorem \ref{th:CP_PP}, but is weaker).

\begin{theorem}\label{th:nonCP_CP}
There exists a recursive $N\unlhd F$ with undecidable $\CPR(F/N)$ and decidable $\CPR(F/N')$.
\end{theorem}

\begin{proof}
C. Miller introduced the following construction in \cite{Miller1}.
Let $U$ be a group given by a finite presentation:
$$
U=\gpr{s_1,\ldots,s_n}{R_1,\ldots,R_m}.
$$
Define a new group $G(U)$ with generators $q,s_1,\ldots,s_n,t_1,\ldots,t_m,d_1,\ldots,d_m$
and relations of several types:
\begin{itemize}
\item
$t_i^{-1} qt_i=qR_i$;
\item
$t_i^{-1} s_kt_i=s_k$;
\item
$d_j^{-1} q d_j=s_j^{-1} q s_j$;
\item
$d_j^{-1} s_k d_j=s_k$;
\end{itemize}
for all $1\le i\le m$, $1\le j\le n$, and $1\le k\le n$.
The group $G(U)$ can be easily recognized as a multiple HNN-extension of a free group
on generators $\{q,s_1,\ldots,s_n\}$ with stable letters $t_1,\ldots,t_m,d_1,\ldots,d_n$.
Hence, $\WP(G(U))$ is decidable (by Britton lemma) by computing reduced forms.
Cyclically reduced forms are computable as well and hence $\PP(G(U))$ is decidable.
Miller proved in \cite{Miller1} that $\CPR(G(U))$ is decidable
if and only if $\WP(U)$ is decidable.
Thus, choosing a finitely presented group $U=\gpr{X}{R}$ with undecidable word problem
we obtain a group $G(U)$ with the required property.
\end{proof}

Examples of finitely presented groups $F/N$ with solvable
word problem and undecidable power problem exist
(can be deduced from \cite[Corollary 1]{Olshanskii_Sapir:1998}). For such groups $\CPR(F/N')$ is undecidable
and $\CPR(F/N'')$ is decidable.
This shows that, in general, both implications of \cite[Problem 12.98(b)]{Kourovka} fail.

\subsection{Some combinatorial problems for groups $F/N^{(d)}$}
\label{se:combin_problems}

In~\cite{Miasnikov-Nikolaev-Ushakov:2014a}, the authors introduce a number of certain
decision, search and optimization algorithmic problems in groups,
such as \emph{subset sum problem},
\emph{knapsack problem}, and \emph{bounded submonoid membership problem}. These problems are
collectively referred to as \emph{knapsack-type} problems and deal with different
generalizations of the classical knapsack and subset sum problems over~$\mathbb Z$
to the case of arbitrary groups. Here we consider the subset sum problem and its generalization
called \emph{acyclic graph membership problem} as defined in \cite{Frenkel-Nikolaev-Ushakov:2014}.

\medskip
\noindent{\bf $\SSP(G)$, subset sum  problem in $G$:}
Given $g_1,\ldots,g_k,g\in G$ decide if
  \begin{equation} \label{eq:SSP-def}
  g = g_1^{\varepsilon_1} \ldots g_k^{\varepsilon_k}
  \end{equation}
for some $\varepsilon_1,\ldots,\varepsilon_k \in \{0,1\}$.

\medskip
\noindent
{\bf $\AGP(G)$, acyclic graph membership problem in $G$:}
Given $g\in G$ and a finite acyclic oriented graph $\Gamma$ labeled by words in $X^\pm$, decide whether there is a directed path in $\Gamma$ labeled by a word $w$ such that $w=g$ in~$G$.

\medskip
These problems were shown to be hard in a vast class of groups.
Below we prove that they are hard in most groups of the type $F/N^{(d)}$.

\begin{theorem}\label{th:SSP_NPcomp}
If $\WP(F/N) \in\P$, $N\ne \{1\}$, and $[F:N]=\infty$, then $\SSP(F/N')$
and hence $\AGP(F/N')$ are $\NP$-complete.
\end{theorem}

\begin{proof}
We claim that under our assumptions the following holds.
\begin{itemize}
\item[(a)]
$N/N'$ is free abelian of infinite rank.
\item[(b)]
It requires polynomial time to find $n$ linearly independent elements in $N/N'$.
\end{itemize}
Clearly, (a) and (b) allow us to reduce zero-one equation problem (ZOE)
known to be $\NP$-complete to $\SSP(F/N')$.

Choose a shortest nontrivial relation $u$ in $F/N$.
By $B_c$ we denote the ball of radius $|u|$ in $F/N$ centered at $c$.
Put $c_1 = 1$. Choose $c_2 \in F/N \setminus B_{c_1}$ and, in general, choose:
$$
c_{n+1} \in F/N \setminus \rb{B_{c_1}\cup\ldots\cup B_{c_n}}.
$$
It is clear from the choice of $c_i$'s that the set of elements:
$$
\{c_i u c_i^{-1} \mid i=1,\ldots,n\},
$$
freely generates a free abelian group of rank $n$ in $F/N'$.
Furthermore, it requires polynomial time in $n$ to construct such a set.
\end{proof}

\begin{proposition}
If $\WP(F/N) \in\P$ and $[F:N]<\infty$, then $\SSP(F/N')$ is in $\P$.
\end{proposition}

\begin{proof}
The condition $[F:N]<\infty$ implies that the group $F/N'$ is virtually abelian (of finite rank).
Then by \cite[Theorem 3.3]{Miasnikov-Nikolaev-Ushakov:2014a}, $\SSP(F/N') \in \P$.
\end{proof}

\begin{theorem}\label{th:SSP_P}
Let $N \unlhd F$.
Then $\SSP(F/N') \in \P$ if and only if $\WP(F/N) \in \P$ and
either $N=\{1\}$ or $[F:N]<\infty$.
\qed
\end{theorem}

\begin{corollary}\label{co:AGP_P}
Let $N \unlhd F$.
Then $\AGP(F/N') \in \P$ if and only if $\WP(F/N) \in \P$ and
either $N=\{1\}$ or $[F:N]<\infty$.
\qed
\end{corollary}

\begin{corollary}
Let $\{1\}\lhd N \unlhd F$ and $\WP(F/N) \in\P$.
Then $\SSP(F/N^{(d)})$ is $\NP$-complete for every $d\ge 2$.
\qed
\end{corollary}

This gives an example of a class of groups with $\NP$-hard subset sum problem which Gromov-Hausdorff limit
has subset sum problem in $\P$.

\end{document}